\documentclass[11pt]{article}
\title{
Translation-finite sets, and weakly compact derivations from $\lp{1}(\Z_+)$ to its dual}
\author{Y. Choi, M. J. Heath%
\thanks{The 2nd author is supported by post-doctoral grant SFRH/BPD/40762/2007 from FCT (Portugal).}
}
\date{\update%
\footnote{MSC2000: Primary 43A20, 43A46, 47B07}%
}

\usepackage[usenames]{color}
\newcommand{\update}{12th October 2009}

\usepackage{amsmath,amsfonts,amssymb}
\numberwithin{equation}{section}

\usepackage{theorem}

\newenvironment{proof}{%
\medskip\noindent{\it Proof}\/.}%
{\hfill$\Box$\break\medskip\ignorespaces}

\newcounter{pulse}
\numberwithin{pulse}{section}
\newtheorem{prop}[pulse]{Proposition}
\newtheorem{lem}[pulse]{Lemma}

\newtheorem{thm}[pulse]{Theorem}
{\theorembodyfont{\upshape}%
\newtheorem{defn}[pulse]{Definition}%

\newtheorem{rem}[pulse]{Remark}
}

\newenvironment{proofof}[1]{%
\medskip\noindent{\it Proof of {#1}\rm}\/.}%
{\hfill$\Box$\break\medskip\ignorespaces}

\usepackage{verbatim}
\newenvironment{asides}{\comment\marginpar{\bf ASIDE}\sf}{\rm\endcomment}

\newenvironment{YCnum}{%
\begin{enumerate}

}{\end{enumerate}\ignorespacesafterend}

\newcommand{\dt}[1]{\emph{#1}\/}  

\renewcommand{\dt}[1]{\textcolor{Bittersweet}{\sf#1}}

\newcommand{\named}[1]{{#1}}  
\renewcommand{\named}[1]{{\sc#1}}  

\newcommand{\Abs}[1]{\left\vert#1\right\vert}
\newcommand{\abs}[1]{\vert#1\vert}
\newcommand{\norm}[1]{\Vert#1\Vert}

\newcommand{\clos}[1]{\overline{#1}}

\newcommand{\lp}[1]{\ell^{#1}}

\newcommand{\Cplx}{{\mathbb C}}
\newcommand{\Nat}{{\mathbb N}}
\newcommand{\Z}{{\mathbb Z}}
\newcommand{\eps}{\varepsilon}

\newcommand{\Bd}{\operatorname{\bf Bd}}

\newcommand{\WAP}{{\sf WAP}} 
\newcommand{\AD}{A({\mathbb D})}  

\newcommand{\sB}{{\sf B}}
\newcommand{\sE}{{\sf E}}

\newcommand{\cV}{{\mathcal V}}

\newcommand{\bbS}{{\mathbb S}}

\begin{document}

\maketitle

\begin{abstract}
We characterize those derivations from the convolution algebra $\ell^1({\mathbb Z}_+)$ to its dual which are weakly compact, providing explicit examples which are not compact. The characterization is combinatorial, in terms of ``translation-finite'' subsets of ${\mathbb Z}_+$, and we investigate how this notion relates to other notions of ``smallness'' for infinite subsets of ${\mathbb Z}_+$. In particular, we prove that a set of strictly positive Banach density cannot be translation-finite; the proof has a Ramsey-theoretic flavour.
\end{abstract}

\section{Introduction}
The problem of determining the weakly compact or compact homomorphisms between various Banach algebras has been much studied; the study of weakly compact or compact derivations, less so.
In certain cases, the geometrical properties of the underlying Banach space play an important role. For instance, by a result of Morris~\cite{SEM_PhD}, every bounded derivation from the \dt{disc algebra} $\AD$ to its dual is automatically weakly compact.
(It had already been shown in~\cite{Bourgain_Acta84} that every bounded operator from $\AD$ to $\AD^*$ is automatically $2$-summing, hence weakly compact; but the proof is significantly harder than that of the weaker result in~\cite{SEM_PhD}.)

In this article, we investigate the weak compactness or otherwise of derivations from the convolution algebra $\lp{1}(\Z_+)$ to its dual. Unlike the case of~$\AD$, the space of derivations is easily parametrized: every bounded derivation from $\lp{1}(\Z_+)$ to its dual is of the form
\begin{equation}\label{eq:main_formula}
D_\psi(\delta_0)=0 \quad\text{and}\quad
 D_\psi(\delta_j)(\delta_k) = \frac{j}{j+k}\psi_{j+k} \qquad(j\in\Nat,k\in\Z_+)
\end{equation}
for some $\psi\in\lp{\infty}(\Nat)$\/.
It was shown in the second author's thesis \cite{MJH_PhD} that $D_\psi$ is compact if and only if $\psi\in c_0$\/, and that there exist $\psi$ for which $D_\psi$ is not weakly compact.

The primary purpose of the present note is to characterize those $\psi$ for which $D_\psi$ is weakly compact (see Theorem~\ref{t:wc-der} below). In particular, we show that there exist a plethora of $\psi$ for which $D_\psi$ is weakly compact but not compact. Our criterion is combinatorial and uses the notion, apparently due to Ruppert, of \dt{translation-finite} subsets of a semigroup. A secondary purpose is to construct various examples of  translation-finite and non-translation-finite subsets of $\Z_+$, to clarify the connections or absence thereof with other combinatorial notions of ``smallness''.

\subsection*{An example}
We first resolve a question from \cite{MJH_PhD}, by giving a very simple example of a $D_\psi$ that is non-compact but is weakly compact.

\begin{prop}\label{p:first-example}
Let $\psi$ be the indicator function of $\{ 2^n \;:\; n\in\Nat \}\subset \Nat$\/. Then $D_\psi$ is non-compact, and the range of $D_\psi$ is contained in $\lp{1}(\Z_+)$\/.

In particular, since $\lp{1}(\Z_+)\subset \lp{p}(\Z_+)$ for every $1<p<\infty$\/, $D_\psi$ factors through a reflexive Banach space and is therefore weakly compact.
\end{prop}

\begin{proof}
Since $\psi\notin c_0$\/, we know by \cite[Theorem 5.7.3]{MJH_PhD} that $D_\psi$ is non-compact.

We have $D_\psi(\delta_0)=0$\/. For each $j\in\Nat$\/, let $N_j=\min(n\in\Nat\;:\; 2^n \geq j)$\/; then 
\[
\norm{D_\psi(\delta_j)}_1  = \sum_{k\geq 0} \Abs{ \frac{j}{(j+k)}\psi_{j+k} } 
	= \sum_{n \geq N_j} \frac{j}{2^n} 
	= \frac{j}{2^{N_j-1}} \leq 2\,.
\]
By linearity and continuity we deduce that $\norm{D_\psi(a)}_1\leq 2\norm{a}_1$ for all $a\in \lp{1}(\Z_+)$.
The last assertion now follows, by standard results on weak compactness of operators. 
\end{proof}

\begin{rem}
Since $D_\psi$ factors through the \emph{inclusion} map $\lp{1}(\Z_+)\to c_0(\Z_+)$\/, which is known to be $1$-summing%
, it too is $1$-summing.
\end{rem}

This last remark raises the natural question: is \emph{every} weakly compact derivation from $\lp{1}(\Z_+)$ to its dual automatically $p$-summing for some $p<\infty$? The answer, unsurprisingly, is negative: we have deferred the relevant counterexample to an appendix.

\section{Characterizing weakly compact derivations}
We need only the basic results on weak compactness in Banach spaces, as can be found in standard references such as \cite{Meggin_intro}.

Recall that if $X$ is a closed subspace of a Banach space $Y$\/, and $K\subseteq X$\/, then $K$ is weakly compact as a subset of $X$ if and only if it is weakly compact as a subset of~$Y$\/.
Since (by Equation~\eqref{eq:main_formula}) our derivations $D_\psi$ take values in $c_0(\Z_+)$, we may therefore work with the weak topology of $c_0(\Z_+)$ rather than that of $\lp{\infty}(\Z_+)$.

Moreover, we can reduce the verification of weak compactness to that of sequential pointwise compactness. This is done through some simple lemmas, which we give below.

\begin{lem}\label{l:weak-is-ptwise-for-bdd-c_0}
Let $(y_i)$ be a bounded net in $c_0(\Z_+)$\/, and let $y\in c_0(\Z_+)$. Then
$(y_i)$ converges weakly to $y$ if and only if it converges pointwise to $y$\/.
\end{lem}

The proof is straightforward and we omit the details.

\begin{asides}
\begin{proof}
Trivially, weak convergence implies pointwise convergence. Suppose, conversely, that $y_i\to y$ pointwise. Put $K=\sup_i \norm{y_i}<\infty$\/. Given $\nu\in\lp{1}(\Z_+)$ and $\varepsilon>0$, choose $M$ large enough that $\sum_{m > M}\abs{\nu_m} <\varepsilon(2K)^{-1}$\/.
Since $y_i\to y$ pointwise, for each $m\in\{1,2,\ldots, M\}$ there exists $i_m$ such that $\abs{(y_i)_m-y_m}\leq \varepsilon {\norm{\nu}_1}^{-1}$ for all $i\succeq i_m$\/.
Choose $j_0$ with $j_0\succeq i_m$ for all $m\in\{1,2,\ldots, M\}$: then, for all $i\succeq j_0$,
\[ \begin{aligned}
\abs{\nu(y_i-y)} 
  &  \leq \sum_{m\geq 0} \abs{\nu_m} \abs{(y_i)_m-y_m}
  &  \leq \sum_{m=0}^M \abs{\nu_m} \frac{\varepsilon}{\norm{\nu}_1} + \sum_{m>M} \abs{\nu_m} \norm{y_i-y}_\infty
  \leq
 2\varepsilon\,.
\end{aligned} \]
Hence $\nu(y_i)\to\nu(y)$ for all $\nu\in\lp{1}(\Z_+)$, as required.
\end{proof}
\end{asides}

\begin{lem}\label{l:checking-wc}
Let $T:\lp{1}(\Z_+) \to c_0(\Z_+)$ be a bounded linear map.
Then the following are equivalent:
\begin{YCnum}
\item\label{item2.2.1} $T$ is weakly compact;
\item\label{item2.2.2} every subsequence of $(T(\delta_n))_{n\in\Nat}$ has a further subsequence which converges pointwise to some $y\in c_0(\Z_+)$\/.
\end{YCnum}
\end{lem}

\begin{proof}
Let $\sB$ denote the closed unit ball of $\lp{1}(\Z_+)$\/, let $\sE=\{T(\delta_n)\;:\;n\in\Nat\}$, and let $\tau$ denote the topology of pointwise convergence in $c_0$\/.
Note that the restriction of $\tau$ to bounded subsets of $c_0$ is a \emph{metrizable} topology.

If \ref{item2.2.1} holds, then by (the trivial half of) Lemma~\ref{l:weak-is-ptwise-for-bdd-c_0}, the bounded set $T(\sB)$ is $\tau$-precompact, and hence sequentially $\tau$-precompact.
Thus~\ref{item2.2.2} holds.

Conversely, suppose that \ref{item2.2.2} holds, i.e. that $\sE$ is sequentially $\tau$-precompact. Then (again by metrizability) we know that $\sE$ is $\tau$-compact, and hence by Lemma~\ref{l:weak-is-ptwise-for-bdd-c_0} it is weakly precompact. Therefore, by 
\named{Krein's theorem} (as it appears in Bourbaki, see \cite[\S IV.5]{Bourbaki_1-5_trans87}), the closed {\bf balanced} convex hull of $\sE$ is weakly compact.
Since this hull is $\clos{T(\sB)}$, $T$ is weakly compact.
\end{proof}

The following notation will be used frequently. Given a subset $S\subseteq \Z_+$ and $n\in\Z_+$, we denote by $S-n$ the set $\{ t\in \Z_+ \;:\; t+n\in S\}$\/.

We need the following definition, due to Ruppert \cite{Ruppert_WAPsets} in a more general setting.
\begin{defn}[{\cite{Ruppert_WAPsets}}]
Let $S\subseteq \Z_+$\/. We say that $S$ is \dt{translation-finite} (\dt{TF} for short) if, for every sequence $n_1<n_2<\ldots$ in $\Z_+$\/, there exists $k$ such that
\begin{equation}
 \bigcap_{i=1}^k (S-n_i )\quad\text{is finite or empty.}
\end{equation}
\end{defn}
(In the later paper \cite{Chou_TAMS90}, TF-sets are called ``\dt{$R_W$ sets}''; we believe that for our purposes the older terminology of Ruppert is more suggestive and apposite.)

TF-sets were introduced by Ruppert in the investigation of weakly almost periodic subsets of semigroups. Recall that a bounded function $f$ on a semigroup $\bbS$ is said to be \dt{weakly almost periodic} if the set of translates $\{s\cdot f \;:\; s\in \bbS\} \cup \{f\cdot s \;:\; s\in \bbS\}$ is weakly precompact in $\ell^\infty(\bbS)$.
Specializing to the case where the semigroup in question is $\Z_+$\/, one of Ruppert's results can be stated as follows.
\begin{thm}[{\cite{Ruppert_WAPsets}}]\label{t:Ruppert}
Let $S\subseteq\Z_+$\/. Then $S$ is TF if and only if all bounded functions $S\to\Cplx$ are WAP as elements of $\lp{\infty}(\Z_+)$\/. In particular, if $S$ is a TF-set then the indicator function of $S$ belongs to $\WAP(\Z_+)$\/.
\end{thm}

It is sometimes convenient to use an alternative phrasing of the original definition (see \cite{Chou_TAMS90} for instance).

\begin{lem}\label{l:alternative-defn}
Let $S\subseteq\Z_+$\/. Then $S$ is non-TF if and only if there are strictly increasing sequences $(a_n)_{n\geq 1},(b_n)_{n\geq 1}\subset \Z_+$ such that $\{a_1,\dots, a_n\}\subseteq S-b_n$ for all~$n$.
\end{lem}

\begin{proof}
Suppose that there exist sequences $(a_n)$\/, $(b_n)$ as described. Then for every $n\geq m\in\Nat$ we have $\{a_1,\ldots, a_m\} \subseteq S-b_n$\/. Hence
\[ \{ b_n \;:\; n\geq m\} \subseteq \bigcap_{k=1}^m (S-a_k) \]
where the set on the left hand side is infinite, for all $m$\/. Thus $S$ is non-TF.

Conversely, suppose $S$ is non-TF: then there is a sequence $a_1<a_2< \ldots$ in $\Z_+$ such that $\bigcap_{j=1}^k (S-a_j)$ is infinite for all $k\in\Nat$\/.
Let $b_1\in S-a_1$\/. We inductively construct a sequence $(b_n)$ as follows: if we have already chosen $b_n$\/, then since $\bigcap_{k=1}^{n+1}(S-a_k)$ is infinite it contains some $b_{n+1}>b_n$\/.
By construction the sequences $(a_n)$ and $(b_n)$ are strictly increasing, and $\{ a_1, \ldots, a_n\} \subseteq S-b_n$ for all $n$\/.
\end{proof}

Our main result, which characterizes weak compactness of $D_\psi$ in terms of $\psi$\/, is as follows.
\begin{thm}\label{t:wc-der}
Let $\psi\in \lp\infty(\Nat)$. Then $D_\psi$ is weakly compact if and only if, for all $\eps>0$\/, the set $S_\eps:=\{n\in\Nat:\abs{\psi_n}>\eps\}$ is TF.
\end{thm}

It is not clear to the authors if one can deduce Theorem~\ref{t:wc-der} in a ``soft'' way from Ruppert's characterization (Theorem~\ref{t:Ruppert}). Instead, we give a direct argument.  The proof naturally breaks into two parts, both of which can be carried out in some generality.

Given $\psi\in\lp{\infty}(\Z_+)$ and $M\in\lp{\infty}(\Z_+\times\Z_+)$\/, define $T^M_\psi: \lp{1}(\Z_+)\to\lp{\infty}(\Z_+)$ by
\begin{equation}
T^M_\psi(\delta_j)(\delta_k) = M_{jk} \psi_{j+k}\qquad(j,k\in\Z_+).
\end{equation}

In particular, if we take
 $\psi\in\ell^\infty(\Nat)$, identified with $(0,\psi_1,\psi_2,\dots)\in\ell^\infty(\Z_+)$, and take
 $M_{jk}$ to be $0$ for $j=k=0$ and to be $j/(j+k)$ otherwise, then $T^M_\psi\equiv D_\psi$\/.

\begin{prop}\label{p:wc-gen-op}
Let $\psi\in\lp{\infty}(\Z_+)$ be such that $S_\eps$ is TF for all $\eps>0$\/.
If $M\in\lp{\infty}(\Z_+\times\Z_+)$ is such that
$\lim_{k\to\infty} M_{jk}=0$ for all~$j$\/,
then $T^M_\psi$ is weakly compact.
\end{prop}

\begin{proof}
To ease notation we write $T_\psi$ for $T^M_\psi$ throughout this proof. Note that the condition on $M$ implies that $T_\psi$ takes values in $c_0(\Z_+)$\/.

Define $\psi_\eps\in\lp\infty(\Z_+)$ as follows:
set $(\psi_\eps)_n$ to be $\psi_n$ if $n\in S_\eps$ and $0$ otherwise.
Then $\psi_\eps$ is supported on $S_\eps$ and $T_{\psi_\eps}\rightarrow T_\psi$ as $\eps\rightarrow 0$.
It therefore suffices to prove that if $\psi$ has TF support, then  $T_\psi$ is weakly compact. 

Let $\psi\in\lp\infty$ be supported on a TF set $S$. Let $(j_n)_{n\geq 1}\subset \Z_+$ be a strictly increasing sequence and set $(j_{0,n})=(j_n)$, $k_0=0$. For each $i\ge 1$ we specify an integer $k_i\in\Z_+$ and a sequence  $(j_{i,n})_{n\geq 1}\subset \Z_+$ recursively, as follows.
If there exists $k\in \Z_+\setminus\{k_0,\dots, k_{i-1}\}$ such that $j_{i-1,n}+k\in S$ for infinitely many $n\in\Nat$\/, let $k_i$ be some such $k$. Otherwise, let $k_i=k_{i-1}$. Let $(j_{i,n})_{n\geq 1}$ be the enumeration of the set 
\[
\{j_{i-1,n}:n\in\Nat, j_{i-1,n}+k_i\in S \}
\] 
with $j_{i,n} <j_{i,n+1}$ for each $n\in\Nat$. Then, by induction on $i$, $(j_{i,n})_n$ is a subsequence of $(j_n)_n$ and, for each $l\in\{1,\dots, i\}$, and each $n\in\Nat$, we have $j_{i,n}+k_l\in S$.

In particular, for each $i\in\Nat$, $\{j_{i,i}+k_1,\dots, j_{i,i}+k_i\}\subset S$\/.
Hence, by our assumption that $S$ is TF, the set $\{k_i:i\in\Z_+\}$ is finite. Let $i_0$ be the smallest $i$ for which $k_i=k_{i+1}$\/: then
for each $k\in \Z_+\setminus \{k_0,\ldots, k_{i_0}\}$\/, there are only finitely many $n\in\Nat$ such that $j_{i_0,n}+k\in S$\/.

Let $K=\{k_0,\ldots, k_{i_0}\}$.
By the \named{Heine-Borel theorem}, there exists a subsequence $(j_{n(m)})_m$ of $(j_{i_0,n})_n$ such that, for each $k\in K$, $\lim_{m\to\infty} T_\psi(\delta_{j_{n(m)}})(\delta_k)$ exists.
Moreover, by the previous paragraph, for each $k\in \Z_+\setminus K$ there exist at most finitely many $m$ such that $j_{n(m)}+k\in S$\/; hence there exists $m(k)$ such that 
\[ T_\psi(\delta_{j_{n(m)}})(\delta_k)= M_{j_{n(m)},k}\psi(j_{n(m)}+k)=0 \quad\text{ for all $m\geq m(k)$\/.} \]
Thus $T_\psi(\delta_{j_{n(m)}})$ converges pointwise to some function supported on $K$\/,
 and the result follows by Lemma~\ref{l:checking-wc}.
\end{proof}

\begin{prop}\label{p:nec-cond}
Let $M\in\lp{\infty}(\Z_+\times\Z_+)$ satisfy
\begin{equation}\label{eq:silly-condition}
\lim_{k\to\infty} M_{jk}=0 \text{ for all $j$\/, and }
   \inf_k \liminf_{j\to\infty} \abs{M_{jk}} = \eta > 0\,.
\end{equation}
Let $\psi\in \lp{\infty}$\/, and suppose that $T^M_\psi$ is weakly compact. Then $S_\eps$ is TF for all $\eps>0$\/.
\end{prop}

\begin{proof}
We first note that \eqref{eq:silly-condition} implies that $T^M_\psi$ has range contained in $c_0(\Z_+)$\/.
Suppose the result is false: then there exists $\eps>0$ such that $S_\eps$ is non-TF.
Hence, by Lemma~\ref{l:alternative-defn}, there exist strictly increasing sequences $(a_n),(b_n)\subset \Z_+$ such that
\[ \{a_1+b_n,\dots, a_n+b_n\}\subset S_\eps\quad\text{ for all~$n$\/.} \]
Now
\[
\abs{T^M_\psi(\delta_{b_n})(\delta_{a_m})}=\abs{M_{b_n,a_m}} \abs{\psi_{a_m+b_n}} \geq \abs{M_{b_n,a_m}}\eps
\quad\text{ for all $m\leq n$\/;}
\]
so, by the hypothesis \eqref{eq:silly-condition}, we have
\begin{equation}\label{eq:ORINOCO}
 \inf_m \liminf_n \abs{T^M_\psi(\delta_{b_n})(\delta_{a_m})} \geq \eta\eps\,. 
\end{equation}

Since $T^M_\psi$ is weakly compact, by Lemma~\ref{l:checking-wc} the sequence $(T_\psi(\delta_{b_n}))$ has a subsequence that converges pointwise to some $\phi\in c_0(\Z_+)$\/. But by \eqref{eq:ORINOCO} we must have $\inf_m \abs{\phi_{a_m}} \geq\eta\eps$\/, so that 
$\phi\notin c_0(\Z_+)$. Hence we have a contradiction and the proof is complete.
\end{proof}

\begin{proofof}{Theorem~\ref{t:wc-der}}
Sufficiency of the stated condition follows from Proposition~\ref{p:wc-gen-op}; necessity, from Proposition~\ref{p:nec-cond}, once we observe that
$\lim_k j/(j+k)=0$ for all~$j$\/, and $\lim_j j/(j+k)=1$ for all~$k$\/. 
\end{proofof}

\begin{rem}
The set $S= \{ 2^n \;:\; n\in\Nat\}$ is TF. In fact, it is not hard to show it has the following stronger property:
\begin{equation}\label{eq:defn-of-T-set}
\text{for \emph{every} $n\in\Nat$\/, the set $S\cap (S-n)$ is finite or empty.} \tag{$\dagger$} 
\end{equation}
We therefore obtain another proof that the derivation constructed in Proposition~\ref{p:first-example} is weakly compact.

Subsets of $\Z_+$ satisfying the condition \eqref{eq:defn-of-T-set} seem not to have an agreed name. They were called \dt{$T$-sets} in work of Ramirez~\cite{Ram_uaFS}, and for ease of reference we shall use his terminology. 
\end{rem}

\section{Comparing the TF-property with other notions of size}
Let $S\subset \Z_+$. For $n\in\Nat$ we define $f_S(n)$ to be the $n$th member of $S$.

\begin{prop}\label{p:clumpy}
Let $S\subset \Z_+$\/. Then there exists a non-TF $R\subset \Z_+$ with $f_R(n)>f_S(n)$ for all $n$.
\end{prop}

\begin{proof}
For $n\in \Z_+$, let $t_n = \frac{1}{2}n(n+1)$ be the $n$th triangular number, so that $t_n=t_{n-1}+n$ for all $n\in\Nat$\/. 
Each $n\in\Nat$ has a unique representation as $n=t_{k-1}+j$ where $k\geq 1$ and $1\leq j\leq k$\/.
Enumerate the elements of $S$ in increasing order as $s_1<s_2<s_3<\ldots$\/, and define a sequence $(r_n)_{n\geq 1}$ by
\[ r_{t_{k-1}+j} = s_{t_k}+j \qquad{(1\leq j\leq k)} \]
as indicated by the following diagram:
\[ \begin{array}{cccc}
 r_1=s_1+1   \\
 r_2=s_3+1 & r_3=s_3+2   \\
 r_4=s_6+1 & r_5=s_6+2 & r_6=s_6+3   \\
 \vdots & & & 
\end{array} \]
Since the sequence $(s_n)$ is strictly increasing,
\[ s_{t_k} \geq s_{t_{k-1}} + (t_k-t_{k-1}) = s_{t_{k-1}}+k \qquad(k\in\Nat), \]
and so the sequence $(r_n)$ is strictly increasing. Put $R=\{ r_n \;:\; n\in\Nat\}$\/: then clearly $f_R(n)=r_n > s_n =f_S(n)$ for all $n$\/.
Finally, since $\bigcap_{j=1}^m (R-j) \supseteq \{ s_{t(k)} \;:\; k\geq m \}$ for all~$m$\/, $R$ is not a TF-set.
\end{proof}

On the other hand, we can find T-sets $S$ such that $f_S$ grows at a ``nearly linear'' rate.

\begin{prop}\label{p:slowgrow_T-set}
Let $g:\Nat\rightarrow\Z_+$ be any function such that $g(n)/n \to \infty$\/.
Then there is a strictly increasing sequence $a_1<a_2<\ldots$ in $\Z_+$, such that $\{a_n:n\in\Nat\}$ is a T-set, while $a_n<g(n)$ for all but finitely many $n$.
\end{prop}

\begin{proof}
Let $N\in\Nat$ and set $k_N$ to be the smallest natural number such that $g(n)>Nn$ for all $n>k_N$. We now construct our sequence $(a_n)$ recursively.
Set $a_0=0$ and suppose that $a_0,\ldots, a_n$ have been defined: if $a_n<k_1$ set $a_{n+1}=a_n+1$; otherwise, if $k_{N}\le a_n<k_{N+1}$ for some $N\in\Nat$, set $a_{n+1}=a_n+N$.
Thus, the elements of $[k_N,k_{N+1}]\cap \{a_n:n\in\Nat\}$ are in arithmetic progression with common difference $N$. 

A simple induction gives that if $k_N\le n<k_{N+1}$, then $a_n\le Nn$. Since for $n\ge k_N$ we also have that $g(n)>Nn$, it follows that, for all  $n\ge k_1$, $a_n<g(n)$. 

Finally, let $i,j\in \Nat$. If $a_i-j\in\{a_n:n\in\Nat\}$ it follows that $a_i< k_{j+1}+j$. Thus, $\{a_n:n\in\Nat\}$ is a T-set.
\end{proof}

\begin{rem}  
Given that  infinite arithmetic progressions are the most obvious examples of non-TF sets, it may be worth noting that if $g(n)/n^2\rightarrow0$, the T-set constructed in the proof of Proposition~\ref{p:slowgrow_T-set} contains arbitrarily long arithmetic progressions.
\end{rem}

\begin{asides}
This example interests me, because of the natural question: ``is the set of primes TF?'' Note that this isn't answered by your result later on  -- YC
\end{asides}

\medskip

The previous two results indicate that the growth of a subset in $\Z_+$ tells us nothing, on its own, about whether or not it is~TF. The main result of this section shows that, nevertheless, certain kinds of \emph{density property} are enough to force a set to be non-TF. First we need some definitions.
\begin{defn}
Let $S\subset\Z_+$\/. The \dt{upper Banach density}\footnotemark\ of~$S$\/, denoted by $\Bd(S)$, is
\[\Bd(S):= \lim_{d\to\infty} \max_n d^{-1} \abs{ S\cap\{n+1,\ldots, n+d\} } \]
(The limit always exists, by a subadditivity argument.)
\begin{asides}
There must surely be a textbook reference to show the limit exists? -- YC
\end{asides}
\end{defn}

\footnotetext{What we call ``Banach density'' is also referred to as \dt{upper Banach density}, and is in older sources given a slightly different but equivalent definition. Some background and remarks on the literature can be found in \cite[\S1]{Hind_lld}, for example.}

For example, the set $R$ constructed in the proof of Proposition~\ref{p:clumpy}
 satisfies
$R \supseteq \{s_{t(m)}+1,\ldots, s_{t(m)}+m\}$ for all $m$,
and so has a Banach density of~$1$.

\begin{prop}\label{p:BD}
Let $S\subset \Z^+$ and suppose $\Bd(S)>0$\/. Then $S$ is not~TF.
\end{prop}
The converse clearly fails: for example, the set $S=\{2^i+j^2:i\in\Nat, j\in\{0,\dots, i\}\}$ is not~TF, but has Banach density zero.

The proof of Proposition~\ref{p:BD} builds on some preliminary lemmas, which in turn require some notation.
Fix once and for all a set $S\subset\Z_+$ with strictly positive
Banach density, and choose $\eps\in(0,1)$ such that $\Bd(S)>\eps$\/.

For shorthand, we say that a subset $X\subseteq \Z_+$ is \dt{recurrent in $S$} if there are infinitely many $n\in\Nat$ such that $X+n\subset S$\/.
For each $d\in\Nat$, let
\[ \cV_d=\{ X\subset\Nat \;:\; \text{$X$ is recurrent in $S$ and } d\geq\abs{X}\geq d\eps\}\/.\]

\begin{lem}\label{l:non-vacuous}
For every $d\in\Nat$, $\cV_d$ is non-empty.
\end{lem}

\begin{proof}
The key step is to prove that the set $\{ i\in\Nat\;:\; \abs{ S\cap\{i+1,\dots, i+d\} }\geq d\eps\}$ is infinite, which we do by contradiction.
For, suppose it is finite, with cardinality~$N$\/, say: then for any $j,n\in\Nat$ we have
\[ \begin{aligned}
(jd)^{-1} \abs{ S\cap\{n+1,\dots, n+(jd)\} }
 & = j^{-1}\sum_{m=0}^{j-1} d^{-1}\abs{ S\cap\{n+md+1,\dots, n+(m+1)d\} }\\
 & \leq  j^{-1} (N+(j-N)\eps)\,,
\end{aligned} \]
so that
\[ \Bd(S)=\lim_j (jd)^{-1} \sup_n \abs{ S\cap\{n+1,\dots, n+(jd)\} }
 \leq \limsup_j j^{-1} (N+(j-N)\eps)=\eps\,, \]
which contradicts our original choice of $\eps$.

It follows that there exists
a strictly increasing sequence $i_1<i_2<\ldots$ in $\Nat$, such that
$\abs{S\cap\{i_n+1,\dots,i_n+d\}}\geq d\eps$ for all $n$\/. Since there are at most finitely many subsets of $\{1,\ldots,d\}$\/, by passing to a subsequence we may assume that the sequence of sets $((S-i_n)\cap\{1,\ldots, d\})_{n\geq 1}$ is constant, with value $X$ say\/. Clearly $X\in \cV_d$, which completes the proof.
\end{proof}

\begin{lem}\label{l:no-orphans}
There exists a sequence $1=d_1<d_2<\ldots$ in $\Nat$ such that, for every $j\in\Nat$ and any $X\in\cV_{d_{j+1}}$\/, there exists $Y\in\cV_{d_j}$ such that
$Y\subseteq X$ and
$\max(Y) < \max(X)$\/.
\end{lem}

\begin{proof}
Put $d_1=1$ and choose $N_1\in\Nat$ such that $N_1 > \eps^{-1}$.
We then inductively construct our sequence $(d_n)$ as follows: if we have already defined $d_j$ for some $j\in\Nat$, let $a_j$ be the largest non-negative integer such that $a_j<d_j\eps$. 
Then choose $N_j\in\Nat$, $N_j\geq 2$, large enough that
\begin{equation}\label{eq:test} 
\frac{1}{N_j}\left[ (N_j-1)\frac{a_j}{d_j} + 1\right] < \eps\,,
\end{equation}
and set $d_{j+1}=N_jd_j$\/.

To show that this sequence has the required properties, let $j\in\Nat$. Given $X\in \cV_{d_{j+1}}$, put $x_0=\min(X)$, and for $m=0,1,\ldots, N_j-1$ put
\[ Y_m= X\cap \{ x_0+md_j, \ldots, x_0+(m+1)d_j-1 \}.\] 
Since $\abs{X}\leq d_{j+1}$ and $\min(X)=x_0$\/, the sets $Y_0,\dots, Y_{N_j-1}$ form a partition of $X$\/.
Since $X$ is recurrent in $S$, so is each of the subsets $Y_m$, and by construction $\abs{Y_m}\leq d_j$ for all~$m$\/.

We claim that there exist $m(1)<m(2)\in\{0,1,\ldots, N_j-1\}$ such that $Y_{m(1)}$ and $Y_{m(2)}$ have cardinality $\geq d_j\eps$\/. If this is the case then
$Y_{m(1)}\in\cV_{d_j}$ and $\max(Y_{m(1)}) < \min(Y_{m(2)}) \leq \max(X)$\/, so that we may take $Y=Y_{m(1)}$ in the statement of the lemma.

Suppose the claim is false. Then at least $N_j-1$ of the sets $Y_0,\ldots, Y_{N_j-1}$ have cardinality $<d_j \eps$, and hence (by the definition of $a_j$) at least $N_j-1$ of these sets have cardinality~$\leq a_j$.
Now since $X\subseteq \{x_0, \ldots, x_0+d_{j+1}-1\}$, we have $X= Y_0\sqcup \ldots \sqcup Y_{N_j-1}$, and so
\[ N_j d_j\eps = d_{j+1}\eps \leq \abs{X} = \sum_{m=0}^{N_j-1} \abs{Y_m} \leq d_j+(N_j-1)a_j.\]
On dividing through by $N_jd_j$, we obtain a contradiction with \eqref{eq:test}, and our claim is proved.
\end{proof}

The final ingredient in our proof of Proposition~\ref{p:BD} is purely combinatorial: it is a version of `\named{K\H{o}nig's infinity lemma}', which we isolate and state for convenience. We shall paraphrase the formulation given in \cite[Lemma~8.1.2]{Die_graph3rd}, and refer the reader to that text for the proof.

\begin{lem}\label{l:Konig}
Let ${\mathcal G}$ be a graph on a countably infinite vertex set $V$, and let $V=\coprod_{j\geq 1} V_j$ be a partition of $V$ into mutually disjoint, non-empty finite subsets.
Suppose that for each $j\geq 1$, every $v\in V_{j+1}$ has a neighbour in $V_j$. Then there exists a sequence $(v_n)_{n\geq 1}$, with $v_n\in V_n$ for each $n$, such that $v_{n+1}$ is a neighbour of $v_n$.
\end{lem}

\begin{proofof}{Proposition~\ref{p:BD}}
Let $(d_j)$ be the sequence from Lemma~\ref{l:no-orphans}. For each $j$\/, let $V_j$ be the set of all subsets of $\{1,\dots, d_j\}$ which are also members of $\cV_{d_j}$\/. The proof of Lemma~\ref{l:non-vacuous} shows that $V_j$ is non-empty, and clearly it is a finite set.

Regard $\coprod_{j\geq 1} V_j$ as the vertex set for a graph, whose edges are defined by the following rule: for each $j\in\Nat$ and $Y\in V_j$\/, $X\in V_{j+1}$, join $X$ to $Y$ with an edge if and only if there exists $m\in\Z_+$ with $Y+m\subseteq X$ and $\max(Y)+m<\max(X)$\/.
Then by Lemma~\ref{l:no-orphans}, every element of $V_{j+1}$ has a neighbour in $V_j$. Hence, by Lemma~\ref{l:Konig}, there exists a sequence $(Y_j)_{j\geq 1}$ of finite subsets of $\Nat$\/, and a sequence $(m_j)\subset \Z_+$\/, such that
\begin{YCnum}
\item $Y_j\in V_j$ for all $j$\/;
\item $Y_j+m_j\subseteq Y_{j+1}$ and $\max(Y_j)+m_j<\max(Y_{j+1})$ for all $j$\/.
\end{YCnum}
Now put $X_1=Y_1$ and put $X_{j+1}= Y_{j+1}-(m_j+\dots +m_1)\subset\Nat$ for $j\geq 1$\/. An easy induction using \emph{both parts} of (ii) above shows that $X_j\subsetneq X_{j+1}$ for all $j$\/.
Since each $Y_i$ is recurrent in $S$\/, so is each $X_i$\/, 
and hence there exist infinitely many $n$ such that $X_i+n\subset S$\/. We may therefore inductively construct $n_1<n_2<\ldots$ such that $X_i+n_i\subset S$ for all $i$\/.

Pick $c_1\in X_1$ and for each $i$ pick $c_{i+1}\in X_{i+1}\setminus X_i$\/; then for all $1\leq i\leq j$ we have $c_i+n_j \in X_j+n_j \subset S$\/;
and since the set $\{ c_i \;:\; i\in\Nat\}$ is infinite, by Lemma~\ref{l:alternative-defn} $S$ is not~TF.
\end{proofof}

\begin{asides}
Note that in using Lemma~\ref{l:no-orphans}, we don't actually need to have lower bounds on the cardinality of the sets $X_i$\/: we only need to know that they form a strictly increasing chain~{--}~YC
\end{asides}

\section{Closing thoughts}
We finish with some remarks and questions that this work raises.
Here and in the appendix, it will be convenient to abuse notation as follows: given $S\subseteq\Nat$\/, we write $D_S$ for the derivation $D_{\chi_S}$, where $\chi_S$ is the indicator function of~$S$. For example, with this notation $D_\Nat \equiv D_{\mathbf{1}}$\/.

\subsection*{Combinatorics of TF subsets of $\Z_+$}
We have been unable to find much in the literature on the combinatorial properties of TF subsets of $\Z_+$. Here are some elementary facts.

\begin{itemize}
\item Let $k\in\Nat$\/; then $S$ is TF if and only if $S+k$ is.
\item Finite unions of T-sets are TF. 
\item The set of odd numbers is non-TF, as is the set of even numbers. In particular, the complement of a non-TF set can be non-TF.
\item Subsets of TF sets are TF. (Immediate from the definition.) In particular, the intersection of two TF sets is TF.
\item If $S$ and $T$ are TF then so is $S\cup T$\/.
\end{itemize}

The last of these observations follows immediately if we grant ourselves Ruppert's result (Theorem~\ref{t:Ruppert} above). It also follows from our Theorem~\ref{t:wc-der}: for if $S$ and $T$ are TF subsets of $\Z_+$, then since $S+1$, $(S\cap T)+1$ and $T+1$ are also TF, the derivations $D_{S+1}$\/, $D_{T+1}$ and $D_{(S\cap T)+1}$ are all weakly compact; whence
\[ D_{(S\cup T)+1}=D_{S+1}-D_{(S\cap T)+1} +D_{T+1}\]
is also weakly compact, so that by the other direction of Theorem~\ref{t:wc-der}, $(S\cup T)+1$ and hence $S\cup T$ are~TF.
It also seems worth giving a direct, combinatorial proof, which to our knowledge is not spelled out in the existing literature (cf.~\cite[Remark 18]{Ruppert_WAPsets}).

\begin{proof}
Let $A_1, A_2$ be subsets of $\Z_+$ and suppose that $A_1\cup A_2$ is not TF. By Lemma~\ref{l:alternative-defn} there exist $a_1<a_2<\ldots$ and $b_1<b_2<\ldots$ in $\Z_+$\/, such that $\{a_m+b_n \;:\; 1\leq m\leq n\}\subseteq A_1\cup A_2$\/.
Let
\[ \begin{aligned}
 E & =\{ (m,n) \in \Nat^2 \;:\; m < n, a_m+b_n\in A_1\}\,, \\
 F & =\{ (m,n) \in \Nat^2\;:\; m < n, a_m+b_n\in A_2\setminus A_1\}\,.
\end{aligned} \]
The sets $E$ and $F$ can be regarded as a partition of the set of $2$-element subsets of $\Nat$. Hence, by \named{Ramsey's theorem}
 \cite[Theorem~9.1.2]{Die_graph3rd},
 there exists either an infinite set $S\subset\Nat$ such that $\{ (x,y) \in S^2 \;:\; x < y \}\subseteq E$\/, or an infinite set $T\subset\Nat$ such that $\{ (x,y)\in T^2 \;:\;x < y\}\subseteq F$\/.

In the former case, enumerate $S$ as $s_0<s_1<s_2<\ldots$\/, and put
$c_j=a_{s_{j-1}}$\/, $d_j=b_{s_j}$ for $j\in\Nat$.
Then $c_i+d_j \in A_1$ for all $1\leq i\leq j$\/; therefore, by Lemma~\ref{l:alternative-defn}, $A_1$ is not TF.

In the latter case, a similar argument shows that $A_2$ is not TF. We conclude that at least one of $A_1$ and $A_2$ is non-TF, which proves the desired result.
\end{proof}

\begin{asides}
Ramsey's theorem seems overkill, but just using one application of the p'hole principle doesn't seem to work -- YC
\end{asides}

\subsection*{Generalizations to other (semigroup) algebras?}
We have relied heavily on the convenient parametrization of ${\mathop{\rm Der}}(\lp{1}(\Z_+),\lp{1}(\Z_+)^*)$ by elements of $\lp{\infty}(\Nat)$.
There are analogous parametrizations for $\Z_+^k$, where $k\geq 2$, but it is not clear to the authors if they allow one to obtain reasonable higher-rank analogues of Theorem~\ref{t:wc-der}.

We can at least make one general observation.

\begin{defn}
Let $A$ be a Banach algebra and $X$ a Banach $A$-bimodule. If $x\in X$\/, we say that $x$ is a \dt{weakly almost periodic element of $X$} if both $a\mapsto ax$ and $a\mapsto xa$ are weakly compact as maps from $A$ to $X$\/. The set of all weakly almost periodic elements of $X$ will be denoted by $\WAP(X)$\/.
\end{defn}

Combining Proposition~\ref{p:nec-cond} with Ruppert's result (Theorem \ref{t:Ruppert}), we see that if $D_\psi$ is weakly compact then $\psi\in\WAP(\ell^\infty(\Z_+))$,
where we identify $\psi\in \ell^\infty(\Nat)$ with $(0,\psi_1,\psi_2,\dots)\in\ell^\infty(\Z_+)$.
This is a special case of a more general result.

\begin{prop}
Let $A$ be a unital Banach algebra, let $D:A\to A^*$ be a weakly compact derivation, and let $\psi\in A^*$ be the functional $D(\cdot)({\sf 1})$\/. Then $\psi\in\WAP(A^*)$\/.
\end{prop}

\begin{proof}
Let $\kappa:A\to A^{**}$ be the canonical embedding. By
\named{Gantmacher's theorem}, $D^*:A^{**}\to A^*$ is weakly compact, so $D^*\kappa$ is also weakly compact. Note that $D^*\kappa(a)=D(\cdot)(a)$ for all $a\in A$\/.

Let $a\in A$\/, and consider $\psi\cdot a \in A^*$\/. For each $b\in A$ we have
\[ (\psi\cdot a)(b) = \psi(ab) =D(ab)({\sf 1}) = D(a)(b) + D(b)(a) \;;\]
thus $\psi\cdot a = D(a)+D^*\kappa(a)$\/. Since $D$ and $D^*\kappa$ are weakly compact, this shows that the map $a\mapsto \psi\cdot a$ is weakly compact. A similar argument shows that the map $a\mapsto a\cdot\psi$ is weakly compact, and so $\psi\in\WAP(A^*)$ as claimed.
\end{proof}

When $A=\lp{1}(\bbS)$ is the convolution algebra of a discrete monoid~$\bbS$, we may regard $A^*=\lp{\infty}(\bbS)$ as an algebra with respect to pointwise multiplication. The previous proposition shows that the functional $D(\cdot)(\delta_e)$ lies in $\WAP(A^*)$\/: is it the case that $h D(\cdot)(\delta_e)$ lies in $\WAP(A^*)$ for every $h\in\lp{\infty}(\bbS)$?

\subsection*{Acknowledgments}
The authors thank the referee and N. J. Laustsen (as editorial adviser) for a close reading of the text and for useful corrections.

\appendix

\section{A weakly compact derivation which is not $p$-summing}
\begin{defn}
Let $X$ and $Y$ be Banach spaces and let $p\in[1,\infty)$. We say that a bounded linear map $T:X\rightarrow Y$ is \dt{$p$-summing} if there exists $C>0$ such that:
\begin{equation}
\label{eq-psum}
\sum_{j=1}^m\norm{Tx_j}^p\le C^p\sup_{\phi\in X^*, \norm{\phi}\le1} \sum_{j=1}^m\abs{\phi(x_j)}^p,\qquad\textrm{for all $m\in\Nat$ and $x_1,\dots, x_m\in X$.}
\end{equation}
The least such $C$ is denoted by $\pi_p(T)$.
If no such $C$ exists 
(i.e. if $T$ is not $p$-summing) we write $\pi_p(T)=\infty$.
\end{defn}

Recall that in Proposition~\ref{p:first-example}, taking $S$ to be the set of integer powers of $2$ gives a derivation $D_S$ that is $1$-summing.
In this appendix, we construct a T-set~$A$ such that $D_A$\/, while weakly compact, is not $p$-summing for any finite~$p$.
To do this, we shall need some standard general results, which are collected in the following proposition for ease of reference.
\begin{prop}\label{p:portmanteau}
Let $X$ and $Y$ be Banach spaces and let $T\in B(X,Y)$\/.
\begin{YCnum}
\item\label{p:psum}
Let $1\leq p \leq q <\infty$\/. Then $\pi_p(T)\ge\pi_q(T)$.
\item\label{item:WC}
If $T$ is $p$-summing for some $p\in[1,\infty)$, then it is weakly compact.
\end{YCnum}
\end{prop}

We refer to~\cite{Die_Jar_Ton} for the proofs: 
part~(i) may be found as \cite[Theorem 2.8]{Die_Jar_Ton}; and part~(ii) follows from the Pietsch factorization theorem, see \cite[Theorem 2.17]{Die_Jar_Ton}.

We now specialize to operators of the form $D_\psi$\/. The key observation is the following.

\begin{lem}\label{l:limbelow}
 Let $\psi\in\ell^\infty(\Nat)$ and $p\in[1,\infty)$ and $K<\pi_p(D_\psi)$. There exists $N\in\Nat$ depending on $\psi$, on $p$ and on~$K$,
 such that for each $\psi'\in\ell^\infty(\Nat)$ satisfying $\psi(k)=\psi'(k)$ for all $k<N$, we have $\pi_p(D_{\psi'})> K$. 
\end{lem}
\begin{proof}
There are $x_1,\dots, x_m\in \ell^1(\Z_+)$ such that 
\begin{equation}
\label{eq-psumFAIL}\sum_{j=1}^m\norm{D_\psi(x_j)}^p> K^p\sup_{\phi\in \ell^\infty, \norm{\phi}\le1}\sum_{j=1}^m\abs{\phi(x_j)}^p.
\end{equation}
Without any loss of generality we may take $x_1,\dots,x_n\in c_{00}$; write $x_j=\sum_{i=0}^{l(j)}\alpha_{i,j}\delta_i$.
For each $j\in\{1,\dots,m\}$, since $D_\psi(x_j)\in c_0$\/, there exists $n(j)\in\Nat$ with 
$\abs{D_\psi(x_j)(\delta_{n(j)})}=\norm{D_\psi(x_i)}$.

Fix $N>\max\{l(1)+n(1),\dots,l(m)+n(m)\}$\/, and let $\psi'\in\ell^\infty(\Nat)$ be such that $\psi(k)=\psi'(k)$ for all $k<N$\/. Observe now that for each~$j$ we have
\[ \begin{aligned}
D_\psi(x_j)(\delta_{n(j)})
 & = \sum_{i=1}^{l(j)}\alpha_{i,j}\frac{i}{i+n(j)}\psi(i+n(j)) \\
 & = \sum_{i=1}^{l(j)}\alpha_{i,j}\frac{i}{i+n(j)}\psi'(i+n(j))
 & =  D_{\psi'}(x_j)(\delta_{n(j)})\,.
\end{aligned} \]
Therefore,
\[\sum_{j=1}^m\norm{D_{\psi'}(x_j)}^p\ge\sum_{j=1}^m\abs{D_{\psi'}(x_j)(\delta_{n(j)})}^p
= \sum_{j=1}^m\abs{D_\psi(x_j)(\delta_{n(j)})}^p
= \sum_{j=1}^m\norm{D_\psi(x_j)}^p.\]
Combining this with Equation \eqref{eq-psumFAIL} yields $\pi_p(D_{\psi'}) > K$\/, and the result follows. 
\end{proof}

We can now give the promised example.
\begin{thm}\label{t:example}
There exists a T-set $A$ such that $D_A$ is not $p$-summing for any $p<\infty$.
\end{thm}
\begin{proof}
The set $A$ will be the disjoint union of a sequence of finite arithmetic progressions whose ``skip size'' tends to infinity.
 For each $k\in\Z_+$, we shall construct, recursively, $A(k)\subset \Nat$ and $c_k\in\Z_+$ such that 
\begin{itemize}
\item[(a)]  $c_k>\max A(k)$;
\item[(b)] $\pi_k(D_B)>k$ for each $B\subset \Nat$ satisfying $B\cap\{1,\dots, c_k\}=A(k)\cap\{1,\dots, c_k\}$;
\item[(c)] $A(k)\supset A(k-1)$ for all $k\geq 1$\/.
\end{itemize}

Let $A(0)=\emptyset$ and $c_0=0$. For each $k\in\Nat$ assume that we have already defined $A(k-1)\subset \Nat$ and $c_{k-1}\in\Nat$ satisfying conditions (a) and (b).
Let $S:=A(k-1)\cup(c_{k-1}+k\Nat)$.
Since $S$ contains an infinite arithmetic progression, it is non-TF. Hence by Theorem~\ref{t:wc-der} $D_S$ is not weakly compact, and so by part \ref{p:psum} of Proposition~\ref{p:portmanteau} it is not $k$-summing.
In particular, $\pi_k(D_S)>k$, so by applying Lemma~\ref{l:limbelow} with $\psi= \chi_S$, we see that there exists $M$ such that
\begin{equation}\label{eq:p-sum-norm}
\pi_k(D_{S\cap\{1,\dots,m\}})>k\quad\textrm{for all $m\geq M$.}
\end{equation}
Choose $n$ such that $c_{k-1}+kn \geq M$\/, and take
\[ A(k):= S\cap\{1,\dots,c_{k-1}+kn\} = A(k-1) \cup\{ c_{k-1}+k, c_{k-1}+2k, \dots, c_{k-1}+n k\}\,.\]
By construction this choice satisfies condition (c). Applying Lemma \ref{l:limbelow} with $\psi= \chi_{A(k)}$\/, we can choose $c_k$ satisfying conditions (a) and~(b), and so our construction may continue.

Set $A:=\bigcup_{k=1}^\infty A(k)$.
Then for each $k\in\Nat$, $A\cap\{1,\dots, c_k\}=A(k)\cap\{1,\dots, c_k\}$ and so $\pi_k(D_A)>k$.
Thus by Propo\-sition~{\ref{p:portmanteau}\ref{p:psum}} $\pi_p(D_A)=\infty$ for all~$p\in[1,\infty)$. Finally, if we
enumerate the elements of $A$ as an increasing sequence $a_1<a_2<\dots$\/,
then $a_{i+1}-a_i\rightarrow\infty$; it follows easily that $A$ is a T-set.
\end{proof}


\vfill

\noindent%
\begin{tabular}{l@{\hspace{20mm}}l}
D\'epartement de math\'ematiques  &
	Departamento de Matem\'atica, \\
\text{\hspace{1.0em}} et de statistique, & \\
Pavillon Alexandre-Vachon &
	Instituto Superior T\'ecnico,  \\
Universit\'e Laval &
	  Av. Rovisco Pais \\
Qu\'ebec, QB &
	 1049-001 Lisboa \\
Canada, G1V 0A6 &
	 Portugal \\
	& \\
{\bf Email: \tt y.choi.97@cantab.net} &
	{\bf Email: \tt mheath@math.ist.utl.pt}
\end{tabular}

\end{document}